\documentclass[reqno,a4paper,11pt,english]{amsart}

\usepackage[utf8]{inputenc}
\usepackage{lmodern}
\usepackage{tikz,tkz-tab}
\usepackage{graphicx}
\usepackage{amsmath,amssymb,amsthm,bbm}
\usepackage{comment}
\usepackage{verbatim}
\usepackage[]{algorithm}
\usepackage{algpseudocodex}
\usepackage{ulem}

\usepackage{enumitem}

\newcommand{\N}{\mathbb N}

\newcommand{\Z}{\mathbb{Z}}

\newcommand{\Zd}{\mathbb{Z}^d}

\renewcommand{\P}{\mathbb{P}}
\newcommand{\E}{\mathbb{E}}

\newcommand{\Ed}{\mathbb{E}^d}

\renewcommand{\epsilon}{\varepsilon}
\renewcommand{\phi}{\varphi}

\renewcommand{\liminf}{\underline{\lim}}

\newcommand{\regine}[1]{\marginpar{}} %{\footnotesize R: #1}
\newcommand{\aurelia}[1]{\marginpar{}} %{\footnotesize A: #1}

\newtheorem{theorem}{Theorem}[section]

\newtheorem{conjecture-dsk}[theorem]{Conjecture}
\newtheorem{lemma}[theorem]{Lemma}
\newtheorem{defi}[theorem]{Definition}

\begin{document}

\title[CPDE]{The Contact Process can survive on a slightly subcritical dynamical percolation cluster}

{
\author{Aurelia Deshayes}
\address{Univ Paris Est Creteil, Univ Gustave Eiffel, CNRS, LAMA UMR8050, F-94010 Creteil, France and IRL CNRS IFUMI-2030, Montevideo, Uruguay.}
\email{aurelia.deshayes@u-pec.fr}
\author{R{\'e}gine Marchand}
\address{%Institut \'Elie Cartan Lorraine (math{\'e}matiques)\\
%Universit{\'e} de Lorraine\\
%Campus Scientifique, BP 239 \\
%54506 Vandoeuvre-l{\`e}s-Nancy  Cedex France\\
Universit\'e de Lorraine, Institut \'Elie Cartan de Lorraine, UMR 7502, Vandoeuvre-l{\`e}s-Nancy, F-54506, France\\
\and \\
CNRS, Institut \'Elie Cartan de Lorraine, UMR 7502, Vandoeuvre-l{\`e}s-Nancy, F-54506, France\\
}
\email{Regine.Marchand@univ-lorraine.fr}
}

\begin{abstract}
The contact process on dynamic edges (CPDE) is a contact process evolving on a dynamic environment given by a dynamical percolation on the edges of $\mathbb Z^d$: each edge updates its state to open or closed with respective rates $vp$ and $v(1-p)$. By coupling a well-chosen subset of once infected sites in the CPDE with a cluster of some supercritical percolation on the edges of $\Zd$, we prove that, for every dimension $d \ge 2$, we can find some slightly subcritical $p<p_c(d)$ such that for every update speed $v>0$, the contact process with large enough infection rate can survive. This extends the result for dimension~1 proved by Linker and Remenik in~\cite{linker2020contact}.
\end{abstract}

\date{\today}

\def\motsclefs{Contact Process, Dynamic Environment, Percolation, Enhancement}

\subjclass[2000]{60K35, 82B43.}
\keywords{\motsclefs}

\maketitle

\section{Introduction}
The standard contact process is a continuous time Markov process introduced by Harris in 1974~\cite{harris1974contact}. It is a fundamental interacting particle system which can be used to model the spread of an infection on a network. Each vertex of a given graph is either infected or healthy and the dynamics are the following ones: infected vertices recover at rate 1 while healthy vertices become infected at rate $\lambda$ times their number of infected nearest neighbors, where $\lambda>0$ is the infection parameter. A key feature of this process is that it exhibits a phase transition on infinite transitive graphs: starting from a single infected vertex, the infection dies out almost surely for small $\lambda$ while it survives with positive probability for large $\lambda$. 

The contact process on a dynamic graph is a natural extension of the classical contact process in which the underlying network evolves over time (see~\cite{valesin_cours} for a recent overview). In addition to the usual infection and recovery mechanisms for vertices, the edges are updated according to an independent stochastic dynamics, modeling the creation and deletion of connections. In this work, we place ourselves within the framework of the Contact Process on Dynamic Edges on $\Z^d$ introduced by Linker and Remenik in~2020~\cite{linker2020contact}. Each edge of $\Z^d$ is now allowed to switch, according to a Markov process, between two states: available or unavailable. This new process has two additional parameters: $p\in(0,1)$ which can be seen as the density of available edges and $v>0$ which is the update speed of an edge. A central issue is to understand how these additional edge dynamics influence the survival of the contact process.

In~\cite{linker2020contact}, many very interesting results about this process have been proven, including the existence of an immunity zone for the parameters $(p,v)$: it is a region in which the contact process cannot survive, no matter how large the infection parameter $\lambda$ is. Linker and Remenik proved that large $p$ can not be in the immunity zone: the dynamics on the links of the network has not a sufficient impact to prevent the infection from surviving. On the contrary, for small enough $p$, there  is a choice of slow speed parameter $v$ such that $(p,v)$ is in the immunity zone: the graph is too sparse for the infection to survive, whatever $\lambda$. The model thus exhibits a new phase transition phenomenon, whose  critical parameter is called $p_1$. The question of the comparison between $p_1$ and $p_c(d)$, the critical parameter of the static edge percolation on $\Z^d$, was left open in~\cite{linker2020contact}.  

In a following paper, Hilario, Ungaretti, Valesin and Vares~\cite{hilario2022results} work on generalized contact processes (which are no longer necessarily Markovian) and one of their results is $p_1\leq p_c(d)$. In this paper, we will prove the strict inequality $p_1<p_c(d)$ for $d\geq 2 $ (it was already proved for $d=1$ in~\cite{linker2020contact}), that is, the contact process can survive in a slighlty subcritical dynamical percolation cluster, no matter how small the update speed $v$ is. This contrasts with the almost sure extinction of the contact process on a static subcritical (thus finite) percolation cluster.

To prove the above mentioned strict inequality, we will use an algorithmic approach, discovering step by step a cluster of infected sites from the origin, and comparing it with a cluster of a new percolation with a well-chosen parameter.

\section{Model and results.}
%%%%%%%%%%%%%%%%%%%%%%%
\subsection{Model}
For $d \ge 2$, we endow the set of vertices $\Zd$ with two sets of edges: $\E_d$ (resp. $\vec{\E}_d$) is the set of unoriented (resp. oriented) edges between sites at euclidean distance one. For $x \in \Zd$, we denote by $\mathcal N(x)$ the set of its $2d$ neighbors.

\medskip
\noindent
\textbf{Stationary dynamic environment, with parameters $(p,v)$}. Let $p \in (0,1)$ and $v>0$ be fixed. The environment with parameters $(p,v)$ is the Feller process $\omega=(\omega_e^t, e \in \E_d)_{t \ge 0}$ such that
\begin{itemize}
\item  $(\omega_e^t)_{t \ge 0}$, $e \in \Ed$, are independent and identically distributed processes.
\item Let $e \in \E_d$ be fixed. The environment $(\omega_e^t)_{t \ge 0}$ at edge $e$ is a continuous-time Markov chain, with values in $\{0,1\}$, switching from state $0$ (\emph{unavailable}) to state $1$ (\emph{available}) at rate $vp$ and from state $1$ to state $0$ at rate $v(1-p)$. Its only invariant law is the Bernoulli law $\mathcal B(p)$ with parameter $p$.
\end{itemize}

The product measure $\mathcal B(p)^{\otimes \E_d}$ is invariant for the environment $(\omega_e^t, e \in \E_d)_{t \ge 0}$. We take as initial condition a family $(\omega_e^0)_{e \in \Ed}$ of independent and identically distributed random variables with law $\mathcal B(p)$, to obtain a stationary environment: at each time $t$, the environment $(\omega_e^t)_{e \in \E_d}$ is a Bernoulli percolation with parameter $p$ on the edges of $\Zd$. Still, the state of a given edge has time correlations. 

%Let us recall that $\mathcal B(p)^{\otimes \E_d}$ is called the percolation measure and 
The percolation event is the event in which the subgraph induced by available edges has an infinite connected component. We then let 
\[
p_c(d)=\inf\{ p \in (0,1): \;  \P (\text{percolation})>0\}.
\]
It is easy to see that $p_c(1)=1$ and it is well-known that $p_c(d)\in(0,1)$ for $d\geq 2$.

\medskip
\noindent
\textbf{Contact process with parameter $\lambda>0$.} In this dynamical percolation environment, we consider a contact process with parameter $\lambda$. More precisely, given the environment $\omega=(\omega_e^t, e \in \E_d)_{t \ge 0}$, we consider the Markov process $(\eta_t)_{t \ge 0}$, taking its values in $\{0,1\}^{\Zd}$, or equivalently in the set of subsets of $\Zd$, such that
\begin{itemize}
\item $\eta_0=\{0\}$.
\item At time $t$, an infected site (i.e. in state $1$) recovers at rate $1$, independently of everything, while a healthy site $x$ (i.e. in state $0$) becomes infected at a rate proportional to the number of its neighbors that are both infected and linked to it by an available edge, i.e. at rate:
$$ \lambda \sum_{y \in \mathcal N(x)} \omega_{\{x,y\}}^t \eta_t(y).$$
\end{itemize}
It can be checked that there indeed exists a càdlàg Feller process $(\{\omega_e^t, e \in \Ed\}, \eta_t)_{t \ge 0}$ with the previous evolution rules: it is called the \emph{contact process with dynamic edges} with parameters $(p,v,\lambda)$, and is denoted $\text{CPDE}(p,v,\lambda)$. Note that the contact process $(\eta_t)_{t \ge 0}$ on its own is no longer Markovian. The CPDE can also be built via a natural adaptation of the classical graphical construction introduced by Harris~\cite{Harris78graphic}, we refer to~\cite{linker2020contact} for further details.

We say that the CPDE$(p,v,\lambda)$ lives forever if, for all $t\geq 0$, $\eta_t\neq \emptyset$ and dies out otherwise. As for the classical contact process, we set 
\begin{align*}
%    \theta(p,v,\lambda) & = \P(\text{$\text{CPDE}(p,v,\lambda)$ lives forever}), \\
%    \lambda_0(p,v) & =\inf\{ \lambda>0: \; \theta(p,v,\lambda)>0\}.
    \lambda_0(p,v) & =\inf\{ \lambda>0: \; \P(\text{$\text{CPDE}(p,v,\lambda)$ lives forever})>0\}.
\end{align*} 

\subsection{Previous results}
%The function $\theta$ is non decreasing in $p$ and $\lambda$ by the natural coupling, and consequently, 
By the natural coupling, the function $\lambda_0$ is non increasing in $p$. The effect of $v$ is much more subtle: increasing $v$ accelerates both the closure and the opening of the edges, and it is unclear whether this facilitates the spread of infection or not. When introducing the model, Linker and Remenik~\cite{linker2020contact} studied the behavior of this process with respect to the parameters $\lambda, v,p$. In particular, \cite{linker2020contact} has highlighted the existence of an immunity zone $\mathcal I$ defined as follows
$$ \mathcal I =\{(p,v) \in (0,1) \times (0,+\infty): \; \lambda_0(p,v)=+\infty \}. $$
In words, if $(p,v)$ is in the immunity zone, then the contact process almost surely dies out, no matter how large the infection parameter is. By monotonicity, it is possible to define
\[
p_1=\inf\{ p \in (0,1): \; \forall v>0, \; \lambda_0(p,v)<+\infty\}.
\] 

We now review some of the results in \cite{linker2020contact} related to the immunity zone and the critical parameter $p_1$.
%existence de la fonction
%Theorem 2.6 of \cite{linker2020contact} says that for all $v>0$, there exists $p_0(v)\in(0,1)$ small enough such that for $p<p_0(v)$, $\lambda_0(v,p)=+\infty$, i.e. the process dies out for all $\lambda$ so $(p,v)$ belongs to the immunity zone. 
Theorem 2.6 says that for every $v>0$, for every $p>0$ small enough, $\lambda_0(v,p)=+\infty$, i.e.  $(p,v)$ belongs to the immunity zone. 
%decroissance de la fonction
Proposition~2.1 establishes that the function $v\to \frac1v \lambda_0(v,p)$ is non increasing in $v$, so in particular, $\lambda_0(v,p)=+\infty$ implies $\lambda_0(v',p')=+\infty$ for all $v'\leq v$, $p'\leq p$. 
%limite en 0
In Theorem 2.7, it is proved that for every $p<1$ large enough, for every $v>0$, $\lambda_0(v,p)<+\infty$, i.e. $(p,v)$ does not belong to the immunity zone.
Finally, if we denote by $\lambda_c(d)$ the critical parameter for the standard contact process on $\Zd$, they proved in Theorem~2.3 that for any $p$, $\lim_{v\to\infty}\lambda_0(v,p)=\lambda_c(d)/p$.
%using that for $v$ large, the update of an edge is fast enough to consider that its states between two moments are almost independent and the CPDE behaves much like a contact process with rate $\lambda p$. 

These results can be summarized in the following diagram:

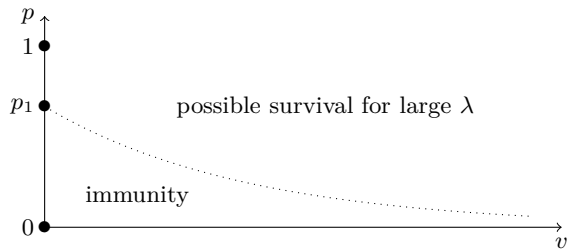
\begin{center}
\begin{tikzpicture}[scale=0.8]
\draw [->] (0,0) -- (8.5,0) ;
\draw (8.5,0) node [below] {{\footnotesize $v$}} ;
\draw [->] (0,0) -- (0,3.5) ;
\draw (0,3.5) node [left] {{\footnotesize $p$}} ;
\draw (0,3) node {$\bullet$} ;
\draw (0,3) node [left] {{\footnotesize $1$}} ;
\draw (0,0) node {$\bullet$} ;
\draw (0,0) node [left] {{\footnotesize $0$}} ;
\draw (0,2) node {$\bullet$} ;
\draw (0,2) node [left] {{\footnotesize $p_1$}} ;
%\draw [domain=0:8] plot (\x, {exp(-\x)});
\draw [dotted,domain=0:8] plot(\x,{2*exp(-0.3*\x)});
\draw (2,2) node [right]{{\footnotesize  possible survival for large $\lambda$}};
\draw (0.5,0.5) node [right]{{\footnotesize immunity}};
\end{tikzpicture} 
\end{center}
  
% LR definit $p_1=\lim_{v\rightarrow0 } p_0(v)$ 

These results were completed by Hilario, Ungaretti, Valesin and Vares in~\cite{hilario2022results}. Among several stronger results, they proved in Theorem 1.1 that for any fixed dimension $d \ge 2$,
\begin{equation}
\label{large_inequality2}
\forall p >p_c(d) \quad \exists \lambda>0 \quad \forall v>0 \quad \P(\text{CPDE}(p,v,\lambda) \text{ lives forever})>0.
\end{equation}
This implies in particular that $p_1 \le p_c(d)$.

\subsection{Result and sketch of proof}
In this work, we prove that the inequality $p_1 \le p_c(d)$ is in fact a strict inequality:

\begin{theorem}
\label{TH-survieSC}
Let $d \ge 2$ be fixed. Then,
%There exists $p<p_c(d)$ such that for every $v>0$, there exists $\lambda>0$ such that $\mathrm{CPDE}(p,v,\lambda)$ on $\Zd$ lives forever with strictly positive probability. 
\begin{equation}
    \label{strict_inequality2}
    \exists p<p_c(d) \quad \forall v>0 \quad \exists \lambda>0\quad \P(\mathrm{CPDE}(p,v,\lambda) \text{ lives forever })>0.
\end{equation}
In other words, $p_1 < p_c(d)$.
\end{theorem}

    Note that the order of quantifiers in~\eqref{strict_inequality2} is weaker than in~\eqref{large_inequality2}. In fact, it is optimal: it is proved in \cite{linker2020contact} (Theorem~2.4) for dimension~1 and \cite{hilario2022results} (Theorem~1.1~i) for dimension $d\geq2$ that
$$\forall p<p_c(d)\quad  \lim_{v\to 0} \lambda_0(v,p)=+\infty,$$ 
so, for every $\lambda$, there exists $v$ small enough such that the contact process dies out.  %multi scale cascading renormalization using the time extinction of the PC in a finite graph, which lead to prove EXTINCTION 

%A block consists in  3 consecutives edges, observed in a time windom with width $T$. A block is said to be good if there is no switch to unavailable state, if there is at least one switch to available state for ech edge, and if there are "a lot of" infections. For $p$ close enough to 1, and for any $v$, it is possible to choose $T$ and $\lambda$ such that the good block has high probability and conclude with a classical comparison with oriented percolation. The fact that $p$ is close to $1$ implies that the density of unavailable edges goes to $0$; it does not seem possible to extend this argument to higher dimension because regardless of how close $p$ is to $p_c$, the density does not go to 0. 

Theorem \ref{TH-survieSC} in dimension 1 is already proved in~\cite{linker2020contact}. The argument is based on a simple block construction, to make a coupling between the CPDE and a two dimensional supercritical oriented percolation on $\Z \times \N$. In dimension 1, $p_c(1)=1$, and the block construction used there relies on the fact that when $p<p_1$ goes to 1, the density of unavailable edges goes to $0$: this cannot be extended to higher dimensions, where $p_c(d)<1$.

To prove $p_1\leq p_c(d)$, the authors in \cite{hilario2022results} also use a block construction to make a coupling between the CPDE and a supercritical oriented percolation. The construction is more complex and is based on the fact that, for a fixed $p>p_c(d)$, the largest connected component of available edges in a finite large box is stable enough to allow survival, regardless of how $v$ is small. Here again, it does not seem possible to extend the argument to $p<p_c(d)$ because the connected components of available edges are, in this case, small or unstable.

Our strategy to prove Theorem \ref{TH-survieSC} is different: we use an algorithmic point of view to discover step by step a subset of the set of at least once infected sites in the CPDE, and we compare it, when the parameters are well chosen, with a non-oriented supercritical percolation cluster on the edges of $\Zd$.

In Section 3, we begin by explaining this strategy in a simpler algorithm, which provides a new proof for \eqref{large_inequality2}.
We take $p>p_c(d)$ and observe that at the first time when an edge is examined to transmit the infection, its state follows the stationary distribution $\mathcal B(p)$, and is independent of every thing else. Based on this observation, our algorithm explores a spanning tree of the connected component of the origin in the standard supercritical percolation with parameter $p$ on the edges of $\Zd$. We orient this spanning tree from the origin to the leafs. The algorithm also checks if the infection succeeds or not along each oriented edge $(x,y)$ of this random tree, by comparing the delay before the next recovery of $x$ and the delay before the next infection through $(x,y)$. Keeping only the successful infections  decimates the spanning tree. If this decimation is not too strong, which is the case if $\lambda$ is large enough, the remaining tree containing the origin spans a new cluster that can be compared to the cluster of a still supercritical percolation. This implies that the built subset of infected sites in the CPDE is infinite with positive probability. As the egdes are examined at most once, the update speed $v$ plays no role here. 

In section 4, to prove Theorem \ref{TH-survieSC}, we need to modify and upgrade the previous algorithm. Indeed, when $p<p_c(d)$, the cluster of the origin in the percolation with parameter $p$ is almost surely finite. To recover some supercriticality, we examine "second chance infections" for edges in some subnetwork: if $(x,y)$ is such an edge and if it is unavailable when we examine it for a first attempt of infection, we simulate the delay before its next switch to the available state (which involves the update speed $v$), and we try a second infection: this gives an extra probability of being open to theses edges. By an enhancement argument, we obtain the needed supercriticality for well chosen parameters.

%%%%%%%%%%%%

\section{Warm up: the supercritical case $p>p_c(d)$}

In this section, we provide an alternative proof for \eqref{large_inequality2}, on which we will elaborate in the next section to prove Theorem \ref{TH-survieSC}. We build a coupling between a process that is stochastically dominated by the $\text{CPDE}(p,v,\lambda)$ and that stochastically dominates the connected component of the origin in a standard Bernoulli percolation on the edges $\E_d$ of $\Zd$. 

\subsection*{First Step.}
In the following Algorithm \ref{ALG-CPDE1-FPP}, we build a process $\underline\eta=(\underline{\eta}_t)_{t \ge 0}$ stochastically dominated by the $\text{CPDE}(p,v, \lambda)$, by considering only infections that succeed at the first attempt to use the corresponding edges, and we also build a coupling with a rooted random tree $T$. Note that the algorithm may not stop, in which case the set $S_T$ of sites treated by the algorithm grows infinitely.

To simplify the analysis, we first introduce all the random variables needed in the algorithm. We consider 3 independent families of independent and identically distributed random variables:  
\begin{itemize}
    \item $(\omega_e)_{e \in \E_d}$, with common law $\mathcal B(p)$, for the states of the edges,
    \item $(X_x)_{x \in \Zd}$, with common law $\mathcal E(1)$, for the delays before recoveries,
    \item $(X_{(x,y)})_{(x,y) \in \vec{\E}_d}$, with common law $\mathcal E(\lambda)$, for the delays before infections.
\end{itemize}
Here is the algorithm:
\begin{algorithm}[h]
\begin{algorithmic}[1]
\Require{$p,\lambda$}
\State $S \gets \{\mathbf{0}\}$, $S_T \gets \varnothing$, $E_0 \gets \varnothing$, $E_1 \gets \varnothing$, $t_I(\mathbf{0}) \gets 0$, $\text{events} \gets \{(t_I(\mathbf{0}),\mathbf{0},+)\}$,
\While{$S  \ne \varnothing$}
    \State Take $x$ in $S$ with some arbitrary deterministic rule, 
%    \State $X_x \sim \mathcal E(1)$, 
    \State $t_R(x) \gets t_I(x)+X_x$, add $\{(t_R(x),x,-)\}$ to $\text{events}$, 
    \For{$y \in \mathcal{N}(x)$ with $y \notin S\cup S_T$} 
%        \State $X_{(x,y)} \sim \mathcal E(\lambda)$, $\omega_{\{x,y\}}\sim \mathcal B(p)$,
        %\If{$\zeta_{(x,y)}=1$}
        \If{$X_{(x,y)}<X_x$ and $\omega_{\{x,y\}}=1$}
            \State $t_I(y) \gets t_I(x)+X_{(x,y)}$, add $y$ to $S$, add $(t_I(y),y,+)$ to $\text{events}$ 
            \State add $(x,y)$ to $E_1$, 
        \Else $\;$ add $(x,y)$ to $E_0$, 
        \EndIf
    \EndFor
    \State move $x$ from $S$ to $S_T$.
\EndWhile
\end{algorithmic}
\caption{The dominated process  and its infection tree}
\label{ALG-CPDE1-FPP}
\end{algorithm}

In Algorithm \ref{ALG-CPDE1-FPP}, we maintain two sets $S$ and $S_T$ of sites, and we build two sets $E_0$ and $E_1$ of oriented edges. The sites in $S_T$ are called "treated", and we say that those in $S \cup S_T$ have been "discovered". 

\begin{itemize}
\item The set $S_T$ contains sites that have been fully treated by the algorithm: we have simulated their infection time $t_I$, their recovery time $t_R$ and some selected outgoing infection edges. The sites in $S_T$ will never be examined again. 
\item The set $S$ contains a list of sites, with known dates of infection, that will finally be treated by the algorithm. 
\item At each passage in the while loop, the algorithm picks a site $x$ in $S$ with some arbitrary deterministic rule. Note in particular that the order in which the algorithm treats events may not be the natural time order. For each neighbor~$y$ not yet discovered, i.e. such that $y\notin S\cup S_T$, we check the lifetime $X_x$,  the delay $X_{(x,y)}$ before the first infection attempt through $(x,y)$ and the state $\omega_{\{x,y\}}$ of the edge. Depending on the results (see line~6 in Algorithm \ref{ALG-CPDE1-FPP}), either the infection succeeds, in which case the edge $(x,y)$ is said to be open, added to $E_1$, and $y$ is added to $S$, or it fails, in which case the edge $(x,y)$ is said to be closed and added to $E_0$.  We end the treatment of $x$ by moving it from $S$ to $S_T$.
\end{itemize}
We then define the process $(\underline{\eta}_t)_{t \ge 0}$ setting
\[
\forall t\ge 0 \quad \forall x\in \Zd \quad x \in \underline{\eta}_t \Longleftrightarrow t_I(x)\le t < t_R(x),
\]
with the convention that $t_I(x)=t_R(x)=\infty$ if these values have not been calculated by the algorithm. Let us now gather some simple observations. 

\begin{enumerate}[label=(P\arabic*), ref=(P\arabic*)]
    \item\label{OBS:infected} The set $S_T$ finally contains all sites that have been infected once. 
    \item\label{OBS:edge1}  There is at most one attempt of infection through any (unoriented) edge.
    \item\label{OBS:edgerate}  When, during one passage in the loop, we treat some $x \in S$, we only examine, between its infection time and its recovery time, the first infection to each neighbor not yet discovered, with rate $\lambda$. 
    \item\label{OBS:site1} Each site is infected and cured at most once.
    %\item\label{OBS:edgeState}The state of the edge $\{x,y\}$ is looked at most once, if an attempt of infection occurs through $(x,y)$ or through $(y,x)$. 
    \item\label{OBS:site}Once a site $y$ is discovered, i.e. enters $S$, the algorithm does not examine other potential infections of $y$ through any edge, so there is only one edge of $E_1$ pointing towards $y$.  
\end{enumerate}
Due to \ref{OBS:edge1}, at the moment when the infection occurs through the edge, the state of edge $\{x,y\}$ follows the law $\mathcal B(p)$ and is independent of everything else. This explains 
why the speed $v$ does not play a role here. 

Due to \ref{OBS:edgerate}, $(\underline{\eta}_t)_{t \ge 0}$ is stochastically dominated by the $\text{CPDE}(p,v, \lambda)$.

Due to \ref{OBS:infected}, if $S_T$ is infinite, then the simulated process $(\underline{\eta}_t)_{t \ge 0}$ lives forever, and so does, by stochastic comparison, the $\text{CPDE}(p,v,\lambda)$ for any $v>0$. 
%\regine{\textcolor{blue}{truc important et caché}}

Let us now examine the random oriented graph $T=(S_T,E_1)$. By construction and because of \ref{OBS:site}, the graph $T$ is a random tree, rooted at the origin and oriented from the root to the leaves: it is the infection tree of finally infected sites. Also by construction and because of \ref{OBS:site}:
\begin{itemize}
    \item[(P6)] each oriented edge from a site in $S_T$ to a site outside $S_T$ is closed.
%    \item[(P7)] an oriented edge between two sites in $S_T$ is either open, either closed or it has never been examined, 
%never    \item[(P8)] no edge between two sites in $S_T^c$ has been examined.
\end{itemize}  

\subsection*{Second step.} We now build a coupling between the infection tree $T$ simulated in Algorithm \ref{ALG-CPDE1-FPP} and the connected component of the origin in a standard Bernoulli percolation on the edges $\E_d$ of $\Zd$.

Let $p \in (p_c(d),1)$  be fixed. We choose $q \in (0,1)$ such that 
\[
pq>p_c(d).
\]
We define random variables that encode the success of the infection attempts:
$$
\forall (x,y) \in \vec{\E}_d \quad \xi_{(x,y)}=\mathbf{1}_{X_{(x,y)}<X_x}.
%\text{ and }\zeta_{(x,y)} = \xi_{(x,y)}\omega_{\{x,y\}}.
$$
Note that $(\omega_e)_{e \in \E_d}$ and $(\xi_{(x,y)})_{(x,y) \in \vec{\E}_d}$ are independent. Moreover $\xi_{(x,y)}$ and $\xi_{(x',y')}$ are independent as soon as $x'\neq x$.
As the $(\xi_{(x,y)})_{(x,y) \in \vec{\E}_d}$ are locally dependent Bernoulli random variables with parameter
\[
\P(\xi_{(x,y)}=1)=\P(X_{(x,y)} <X_x) = \frac{\lambda}{\lambda+1} \stackrel{\lambda \to +\infty}{\to} 1,
\]
we can apply the stochastic comparison result by Liggett, Schonmann and Stacey~\cite{Liggett_DomStoch}: for every $0<q<1$, we can choose $\lambda$ large enough and by increasing the probability space if necessary, we can consider on the same probability space a second family of independent and identically distributed random variables, 
\begin{equation}
\label{EQ:domstoch}
\text{$(\underline \xi_{(x,y)})_{(x,y) \in \vec{\E}_d}$, with law $\mathcal B(q)$, such that } \forall (x,y) \in \vec{\E}_d \quad \underline \xi_{(x,y)} \le \xi_{(x,y)} \; a.s. 
\end{equation}
and independent from $(\omega_e)_{e \in \E_d}$.

\smallskip
We now build a coupling between the previous set $S_T$ of sites once infected in $\underline \eta$ and the 
connected component of the origin of an independent Bernoulli percolation on edges $\E_d$ of $\Zd$ with parameter $pq$. More precisely, we remove from Algorithm~\ref{ALG-CPDE1-FPP} the simulation of the $t_I$'s and $t_R$'s (which allowed us to construct the process $\underline \eta$), we only keep the simulation of the infection tree $T=(S_T,E_1)$, and we add the construction of another coupled tree $(\underline S_T, \underline E_1)$.

\begin{algorithm}[h]
\begin{algorithmic}[1]
\Require{$p,\lambda$}
\State {$S \gets \{\mathbf{0}\}$, $S_T \gets \varnothing$}, {$E_0 \gets \varnothing$, $E_1 \gets \varnothing$}, 
\State $\underline S \gets \{\mathbf{0}\}$, $\underline S_T \gets \varnothing$,
 $\underline E_0 \gets \varnothing$, $\underline E_1 \gets \varnothing$, 
\While{$S \cup \underline S \ne \varnothing$}
    \State Take $x$ in $S\cup \underline S$ with some arbitrary deterministic  rule, 
    \If{\textcolor{blue}{$x \in S$}}
        \For{\textcolor{blue}{$y \in \mathcal{N}(x)$ with $y \notin S\cup S_T$}} 
            \If{\textcolor{blue}{$\xi_{(x,y)}\omega_{\{x,y\}}=1$}}
                \textcolor{blue}{add $y$ to $S$ and $(x,y)$ to $E_1$}, 
            \Else 
                $\;$ \textcolor{blue}{add $(x,y)$ to $E_0$}, 
            \EndIf
        \EndFor
        \State \textcolor{blue}{move $x$ from $S$ to $S_T$},
    \EndIf
    \If{$x \in \underline S$}
        \For{$y \in \mathcal{N}(x)$ with $y \notin \underline S\cup \underline S_T$}
            \If{$\underline \xi_{(x,y)}\omega_{\{x,y\}}=1$}
                add $y$ to $\underline S$ and $(x,y)$ to $\underline E_1$, 
            \Else 
                $\;$ add $(x,y)$ to $\underline E_0$, 
            \EndIf
        \EndFor
        \State move $x$ from $\underline S$ to $\underline S_T$.
    \EndIf
\EndWhile
\end{algorithmic}
\caption{Coupling between the set $S_T$ of once infected sites and independent Bernoulli percolation}
\end{algorithm}

First, note that if we omit the first "if" block (in blue), this simulates a Markovian construction and exploration algorithm of the connected component $\underline S_T$ of the origin in the standard independent Bernoulli percolation on the edges of $\Zd$ with parameter
$$\P(\underline \xi_{(x,y)}\omega_{\{x,y\}}=1) = \P(\underline \xi_{(x,y)})\P(\omega_{\{x,y\}}=1)=pq.$$
Moreover, for the same reasons as for $T$, the random graph $\underline T =(\underline S_T, \underline E_1)$ is a spanning tree of this connected component. 

Let us now check the following property of the coupling: at each step, 
\begin{equation}
\label{EQ:inclu}
    \underline S_T \subset S_T \text{ and } \underline S \cup \underline S_T \subset S \cup S_T.
\end{equation}
We proceed by recurrence. It is obviously true before entering the while loop.
Let us assume now that $\underline S_T \subset S_T \text{ and } \underline S \cup \underline S_T \subset S \cup S_T$ when we begin a passage in the while loop. We take $x \in S\cup \underline S$.
\begin{itemize}
    \item First case: $x \in S\cap \underline S$. Consider $y \in \mathcal{N}(x)$.
    As $\underline S \cup \underline S_T \subset S \cup S_T$, we have two cases:
    \begin{itemize}
    \item either $y \notin S \cup S_T$, and thus $y \notin \underline S \cup \underline S_T$. The edge $(x,y)$ is examined twice, and as $\P(\underline \zeta_{(x,y)} \le \zeta_{(x,y)})=1$, if $y$ is added to $\underline S$, it is also added to $S$, which preserves \eqref{EQ:inclu}.
    \item Or $y \in S \cup S_T$ and $y \notin \underline S \cup \underline S_T$. The site $y$ may be added to $\underline S$, but as it is already in $S \cup S_T$,  \eqref{EQ:inclu} is preserved.
    \end{itemize}
    At the end of the passage, the two moves of $x$ from $S$ to $S_T$ and from $\underline S$ to $\underline S_T$ preserve \eqref{EQ:inclu}.
    \item Second case: $x \in S$ and $x \notin \underline S$. The neighbors of $x$ may only be added to $S$, which preserves \eqref{EQ:inclu}. At the end of the passage, the move of $x$ from $S$ to $S_T$ preserves \eqref{EQ:inclu}.
    \item Third case: $x \notin S$ and $x \in \underline S$. As $\underline S \cup \underline S_T \subset S \cup S_T$, $x \in S_T$. Consider $y \in \mathcal{N}(x)$ and imagine it is added to $\underline S$: so $\underline \zeta_{(x,y)}=1$, which implies $\zeta_{(x,y)}=1$. As $x \in S_T$, this implies that $y$ is already in $S\cup S_T$ and its addition to $\underline S$ preserves~\eqref{EQ:inclu}. At the end of the passage, $x$ is moved from $\underline S$ to $\underline S_T$. As $x \in S_T$, this move  preserves \eqref{EQ:inclu}.
\end{itemize}
This ends the proof of \eqref{EQ:inclu}. Note that it may happen that $\underline S \not\subset S$ and $\underline E_1 \not\subset E_1$.

\medskip
We can now conclude the proof of \eqref{large_inequality2}. Let us recall that $p \in (p_c(\Zd),1)$ is fixed and we have chosen $q \in (0,1)$ such that 
$pq>p_c(\Zd)$. We have constructed~$\underline{S}_T$ which is the connected component of theorigin in a supercritical percolation of parameter~$pq$ so, with positive probability, $\underline S_T$ grows infinitely. We have also chosen $\lambda$ large enough so that stochastic domination~\eqref{EQ:domstoch} holds and, consequently, the coupling inclusion~\eqref{EQ:inclu} implies that $S_T$ also grows infinitely with positive probability. Since, by definition, $S_T$ is the set of once infected sites in the process $\underline \eta$, this process lives forever with positive probability. By stochastic comparison, the $\text{CPDE}(p,v,\lambda)$ lives forever with positive probability for any $v>0$. 

%%%%%%%%%%%%

\section{Slightly subcritical case, with second chance infections}

In this section, we prove Theorem \ref{TH-survieSC}.
Our proof, as the one for \eqref{large_inequality2}, is based on an algorithm that builds a coupling between a process stochastically dominated by the $\text{CPDE}(p,v,\lambda)$ and  that stochastically dominates the connected component of the origin in a standard Bernoulli percolation on the edges $\E_d$ of $\Zd$.
The percolation process will now have two types of edges, corresponding to two types of infections: the first chance ones, which are the same as in the previous section, and some new ones that we call second chance infections.

\subsection*{First step} We begin by simulating a process $\underline\eta$ stochastically dominated by the $\mathrm{CPDE}(p,v,\lambda)$, and by building its infection tree. We proceed by modifying and upgrading Algorithm \ref{ALG-CPDE1-FPP}. As before, each (unoriented) edge is treated at most once, but the treatment is more complex. 
Imagine that, at the beginning of a passage in the while loop, we take some $y \in S$. For each neighbor~$z$ not yet discovered, we examine the state $\omega_{\{y,z\}}$ of the unoriented edge $\{y,z\}$ at time $t_I(y)$  and we separate the cases:
\begin{itemize}
    \item If $\omega_{\{y,z\}}=1$, we simulate the next infection time of $z$ through the edge $(y,z)$, and the next switch to unavailable state of the edge $\{y,z\}$.
If the infection occurs before the recovery of $y$ and before the switch of $\{y,z\}$, it succeeds, we add $z$ to $S$ and $(y,z)$ to $E_1$ as before, otherwise it fails and we add $(y,z)$ to $E_0$.
\item If $\omega_{\{y,z\}}=0$ and if $\{y,z\}$ is in some subset $\vec{\E}_d^s$ of special edges defined below, we give it a second chance. We wait for the next change of the edge $\{y,z\}$ to available and we try the next infection through $\{y,z\}$: if it succeeds we add $(y,z)$ to $E_2$ and $z$ to $S$, otherwise we add $(y,z)$ to $E_0$. We distinguish between two subcases that will also appear in the forthcoming definition of the random variable encoding this second chance infection:
    \begin{itemize}
        \item[\eqref{indic_second_chance1}] Either $\{y,z\}$ switches to available before $y$ recovers, and we can examine if the next infection through $(y,z)$ succeeds, i.e. if it happens before both the recovery of $y$ and the next switch of $\{y,z\}$ back to unavailable.
        \item[\eqref{indic_second_chance2}] Or $y$ recovers before $\{y,z\}$ switches to available. In that case, we do the following. At time $t_I(y)$, the site $y$ has just been infected from its parent, say $x$. In particular, the edge $\{x,y\}$ is available. Imagine that the edge $\{y,z\}$ changes to available before the edge $\{x,y\}$ changes to unavailable, and that the standard contact process restricted to the edge $\{x,y\}$ and its two extremities is still alive with $y$ infected at that time: then we can examine again if the next infection through $(y,z)$ succeeds, i.e. if it happens before both the next recovery of $y$ and the next switch of $\{y,z\}$ back to unavailable.
    \end{itemize}
\end{itemize}

One might think that case ~\eqref{indic_second_chance1} would suffice. But if the update velocity $v$ is very small, the probability that $y$ is still infected when $\{y,z\}$ switches to the available state is very low, regardless of the value of $\lambda$. On the other hand, when $\lambda$ is large, the standard contact process $\tilde \eta_{\{x,y\}}$ has a high probability to still be infected when $\{y,z\}$ switches to the available state. That is why we need~\eqref{indic_second_chance2}.

As we will need some independence, we restrict these second chances to edges in
$$\vec{\E}_d^s =\{(x,x+e_1): \; x \in (1+4\Z)^d\} \cup \{(x+e_1,x): \; x \in (1+4\Z)^d\}.$$
\begin{comment}
\textcolor{blue}{Régine. Remarques pour le choix final de $\E_d^r$
\begin{itemize}
    \item prendre une notation qui suggère l'orientation des arêtes
    \item $0$ ne doit pas être l'origine d'une arête d'$\E_d^r$, car il n'a pas de parent
    \item un point $y$ origine d'une arête d'$\E_d^r$ ne doit pas être la destination d'une arête d'$\E_d^r$, car il doit avoir un parent, et c'est plus simple que le parent résulte d'une infection de première chance.
    \ietm on veut: $(\zeta_{(y,z)}^x)_{x,z}$ and $(\zeta_{(y',z)}^{x})_{x,z}$ are independent as soon as $y\ne y'$.
    \item invariance par "certaines translations", pour pouvoir faire de l'enhencement.
\end{itemize}}
\end{comment}
As in the previous section, we begin by introducing all the random variables needed for the algorithm. We consider the following independent families of independent and identically distributed random variables: 

\smallskip
\noindent
\underline{For the dynamic environment:}
\begin{itemize}
    \item $(\omega_e)_{e \in \E_d}$, with common law $\mathcal B(p)$, for the states of the edges,
    \item $(T^+_e)_{e \in \E_d}$, with common law $\mathcal E (vp)$, for the switches to available state,
    \item $(T^-_e)_{e \in \E_d}$, with common law $\mathcal E (v(1-p))$, for the switches to unavailable state.
\end{itemize}

\noindent
\underline{For the infection and recovery attempts:}
\begin{itemize}
    \item $(X_x)_{x \in \Zd}$ and $(X_x')_{x \in \Zd}$ , with common law $\mathcal E(1)$, for delays before recoveries,
    \item $(X_{(x,y)})_{(x,y) \in \vec{\E}_d}$, with common law $\mathcal E(\lambda)$, for delays before infections,
    \item $(\chi_x)_{x \in \Zd}$, i.i.d Poisson point processes with intensity $1$, for recovery marks,
    \item $(\chi_{\{x,y\}})_{\{x,y\} \in \E_d}$, i.i.d Poisson point processes with intensity $\lambda$ , for infection marks.
\end{itemize}
We then define new collections of random variables. In the first family, $\xi_{(y,z)}$  encodes if the next infection attempts through the edge $(y,z)$ occurs before the recovery of the infected extremity $y$ and before the next switch of the edge to the unavailable state:
\begin{align*}
    \forall (y,z) \in \vec{\E}_d \quad \xi_{(y,z)} & =\mathbf{1}_{X_{(y,z)}<\min(X_y,T^-_{\{y,z\}})}.
\end{align*}
Next, for every $(x,y) \in \vec{\E}_d$, we define a process $(\tilde \eta_t^{x,y})_{t \ge 0}$:
\begin{itemize}
    \item If $X_{(x,y)}<X_x$, $(\tilde \eta_t^{x,y})_{t \ge 0}$ is the standard contact process restricted to the edge $\{x,y\}$, starting with both sites infected, and using 
    \begin{itemize}
        \item[$\circ$] $\{X_x-X_{(x,y)}\} \cup (X_x-X_{(x,y)}+\chi_x)$ as recovery marks for $x$, 
        \item[$\circ$] $\{X_y\} \cup (X_y+\chi_y)$ as recovery marks for $y$, 
        \item[$\circ$] $\chi_{\{x,y\}}$ as infection marks.
    \end{itemize}
    Thanks to the Markov property, the random variables have the required distributions.
    \item If $X_{(x,y)}>X_x$, $(\tilde \eta_t^{x,y})_{t \ge 0}$ is the empty process.
\end{itemize}
Finally, in the last collection, the random variable $\zeta_{(y,z)}^x$ encodes the success of the second chance infection attempt through $(y,z) \in \vec{\E}_d^s$ with the help of $y$'s neighbor~$x$: if $(y,z) \in \vec{\E}_d^s$, for every $x \in\mathcal{N}(y)\setminus\{z\}$, 
\begin{align}
\zeta_{(y,z)}^x & = %\mathbf{1}_{T^+_{\{y,z\}}  <T^-_{\{x,y\}} \mathcolor{teal}{-X_{(x,y)}}} 
\mathbf{1}_{T^+_{\{y,z\}}  < X_y}\mathbf{1}_{X_{(y,z)}<\min(X_y-T^+_{\{y,z\}},T^-_{\{y,z\}})} \label{indic_second_chance1}\tag{$\star$} \\
& + \mathbf{1}_{X_y<T^+_{\{y,z\}}} \mathbf{1}_{T^+_{\{y,z\}}  <T^-_{\{x,y\}}-X_{(x,y)}}  \mathbf{1}_{ y \in  \tilde \eta^{x,y}_{T^+_{\{y,z\}}}} \mathbf{1}_{X_{(y,z)}<\min(X_y',T^-_{\{y,z\}})}.\label{indic_second_chance2}\tag{$\star\star$}
\end{align}
See Figure \ref{FIG1} for an illustration of \eqref{indic_second_chance2}.
%    \item If $(y,z) \notin \E_d^r$, for every $x \in\mathcal{N}(y)\setminus\{z\}$, 
%    $ \zeta_{(y,z)}^x = 0.$
%    \item To simplify the presentation of the algorithm, we say that the parent of $\mathbf{0}$ is $\mathbf{0}$, and we define for every $x \in\mathcal{N}( \mathbf{0})$, $\zeta_{(0,0)}^x=0$.
%To simplify the presentation of the next algorithm, 
%we also set $\zeta_{(y,z)}^x = 0$ when $(y,z) \in \E_d^r$ and $x=z$, when $(y,z) \notin \E_d^r$ and $x \in \mathcal N(y)$, and when $y=z=0$ and $x \in \mathcal N(0)$.

\begin{figure}[h]
    \centering
    \begin{tikzpicture}[scale=0.8]
\draw [red!10, fill = red!10] (-1,-2) -- (0,-2) -- (0,3.2) -- (-1,3.2);
\draw [red!20, fill = red!20] (-1,0) -- (0,0) -- (0,2) -- (-1,2);
\draw [red!10, fill = red!10]  (0,2) -- (1,2) -- (1,4.5) -- (0,4.5);
%\draw [-] (-3,0) -- (3,0) ;
\draw [-] (-3,-2) -- (3,-2) ;
\draw [dotted] (0,0) -- (3,0) ;
\draw [dotted] (-4.5,-2) -- (-3,-2) ;
%\draw (-3,0) node [below left] {{\footnotesize $0$}} ;
\draw [dotted] (0,2) -- (4,2) ;
\draw [<->, dotted] (2,0) -- (2,2) ;
\draw (2,2) node [below right] {{\footnotesize $T^+_{\{y,z\}}$}} ;
\draw [dotted] (-3.5,3.2) -- (-1,3.2) ;
\draw [<->, dotted] (-3.5,-2) -- (-3.5,3.2) ;
\draw (-3.5,3.2) node [below left] {{\footnotesize $T^-_{\{x,y\}}$}} ;
\draw [dotted] (-4,1) -- (-1,1) ;
\draw [<->, dotted] (-4,-2) -- (-4,1) ;
\draw (-4,1) node [below left] {{\footnotesize $X_x$}} ;
\draw [->] (-3,-2) -- (-3,5) ;
\draw (-3,5) node [left] {{\footnotesize $t$}} ;
\draw [dotted] (0,-2) -- (0,5) ;
\draw (0,-2) node {$\bullet$} ;
\draw (0,-2) node [below] {$y$} ;
\draw [dotted] (1,-2) -- (1,5) ;
\draw (1,-2) node {$\bullet$} ;
\draw (1,-2) node [below] {$z$} ;
\draw [dotted] (-1,-2) -- (-1,5) ;
\draw (-1,-2) node {$\bullet$} ;
\draw (-1,-2) node [below] {$x$} ;
\draw [->, red, thick] (-1,0) -- (0,0) ;
\draw [red, thick] (0,0) -- (0,0.4) ;
\draw (0,0.4) node [red] {$\times$} ;
\draw [dotted] (0,0.4) -- (2.5,0.4) ;
\draw [<->, dotted] (2.5,0) -- (2.5,0.4) ;
\draw (2.5,0.4) node [below right] {{\footnotesize $X_y$}} ;
\draw [->, red, thick] (-1,0.7) -- (0,0.7) ;
\draw [dotted] (-4.5,0) -- (-1,0) ;
\draw [<->, dotted] (-4.5,-2) -- (-4.5,0) ;
\draw (-4.5,0) node [below left] {{\footnotesize $X_{(x,y)}$}} ;
\draw [red, thick] (-1,-2) -- (-1,1) ;
\draw (-1,1) node [red] {$\times$} ;
\draw [->, red, thick] (0,1.4) -- (-1,1.4) ;
\draw [red, thick] (0,0.7) -- (0,1.6) ;
\draw (0,1.6) node [red] {$\times$} ;
\draw [->, red, thick, dotted] (0,1.1) -- (1,1.1) ;
\draw [red, thick] (-1,1.4) -- (-1,2) ;
\draw [->, red, thick] (-1,1.75) -- (0,1.75) ;
\draw [red, thick] (0,1.75) -- (0,2.9) ;
\draw (0,2.9) node [red] {$\times$} ;
\draw [->, red, thick] (0,2.4) -- (1,2.4) ;
\draw [dotted] (1,2.4) -- (4,2.4) ;
\draw [<->, dotted] (4,2) -- (4,2.4) ;
\draw (4,2.4) node [below right] {{\footnotesize $X_{(y,z)}$}} ;
\draw [dotted] (0,2.9) -- (3.5,2.9) ;
\draw [<->, dotted] (3.5,2) -- (3.5,2.9) ;
\draw (3.5,2.9) node [right] {{\footnotesize $X'_y$}} ;
\draw [dotted] (1,4.5) -- (4,4.5) ;
\draw [<->, dotted] (3,2) -- (3,4.5) ;
\draw (3,4.5) node [below right] {{\footnotesize $T^-_{\{y,z\}}$}} ;
\end{tikzpicture} 
    \caption{The first chance infection through $(x,y)$, encoded in $\xi_{(x,y)}$, succeeds, then the second chance infection through $(y,z)$ with the help  of $y$'s neighbor $x$, encoded by $\zeta^x_{(y,z)}$, succeeds thanks to a type $(\star\star)$ event. The time slots when edges are available are in pink. The contact process $\tilde \eta_{\{x,y\}}$ is simulated during the darker slot.}
    \label{FIG1}
\end{figure}
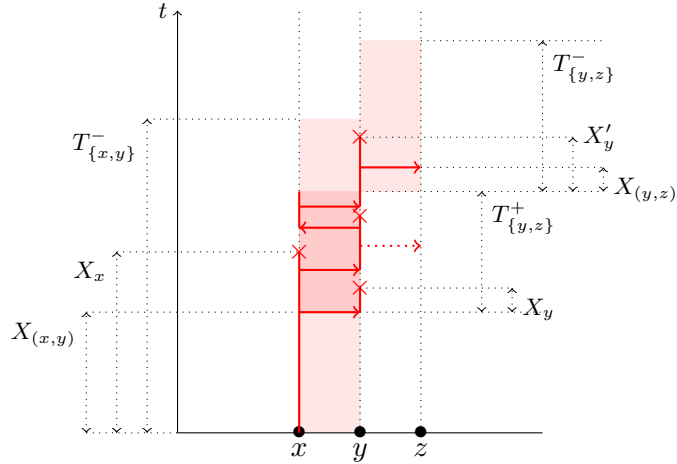
Note also that in the latter case, we use a new variable $X_y'$ for the recovery time of $y$ and we use $X_{(y,z)}$ as the delay before the next infection attempt after $T^+_{\{y,z\}}$: thanks to the Markov property, each one has the required distribution.

We are now ready to present the new algorithm:
\begin{algorithm}[H]
\begin{algorithmic}[1]
\Require{$p,v,\lambda$}
\State $S \gets \{\mathbf{0}\}$, $S_T \gets \varnothing$,  $E_0 \gets \varnothing$, $E_1 \gets \varnothing$, $E_2 \gets \varnothing$, $t_I(\mathbf{0})=0$, 
\State $\text{events} \gets \{t_I(\mathbf{0}),\mathbf{0},+)\}$, $\text{parent}(\mathbf{0}) \gets \mathbf{0}$,
\While{$S \ne \varnothing$}
    \State take $y$ in $S $ with some arbitrary deterministic rule, $x \gets \text{parent}(y)$, 
       \State $t_R(y) \gets t_I(y)+X_y$, add  $(t_R(y),y,-)$ to $\text{events}$,
        \For{$z \in \mathcal N(y)$ with $z \notin S \cup S_T$}
            \If{$\omega_{\{y,z\}}=1$}
                \If{$\xi_{(y,z)}=1$}
                   \State $t_I(z) \gets t_I(y)+X_{y,z}$, add $z$ to $S$, add $(t_I(z),z,+)$ to $\text{events}$, 
                  \State add $(y,z)$ to $E_1$, $\text{parent}(z) \gets y$,
                \Else $\;$ add $(y,z)$ to $E_0$, 
                \EndIf
            \ElsIf{$(y,z) \in \vec{\E}_d^s$ and $\zeta^x_{(y,z)}=1$}
                \State $t_I(z) \gets t_I(y)+T^+_{\{y,z\}}+X_{(y,z)}$, add $z$ to $S$,
                \State add $(t_I(z),z,+)$ to $\text{events}$, add $(y,z)$ to $E_2$, {$\text{parent}(z) \gets y$},
            \Else $\;$ add $(y,z)$ to $E_0$,
            \EndIf
        \EndFor  
        \State move $y$ from $S$ to $S_T$
\EndWhile
\end{algorithmic}
\caption{The second dominated process and its infection tree}
\label{ALG-CPDE2}
\end{algorithm}

\begin{comment}
    \begin{algorithm}[H]
\begin{algorithmic}[1]
\Require{$p,v,\lambda$}
\State $S \gets \{\mathbf{0}\}$, $S_T \gets \varnothing$,  $E_0 \gets \varnothing$, $E_1 \gets \varnothing$, $E_2 \gets \varnothing$, $t_I(\mathbf{0})=0$, 
\State $\text{events} \gets \{t_I(\mathbf{0}),\mathbf{0},+)\}$, $\text{parent}(\mathbf{0}) \gets \mathbf{0}$,
\While{$S \ne \varnothing$}
    \State take $y$ in $S $ with some arbitrary deterministic rule, 
       \State $t_R(y) \gets t_I(y)+X_y$, add  $(t_R(y),y,-)$ to $\text{events}$,
        \For{$z \in \mathcal N(y)$ with $z \notin S \cup S_T$}
            \If{$\omega_{\{y,z\}}=1$}
                \If{$\xi_{(y,z)}=1$}
                   \State $t_I(z) \gets t_I(y)+X_{y,z}$, add $z$ to $S$, add $(t_I(z),z,+)$ to $\text{events}$, 
                  \State add $(y,z)$ to $E_1$, $\text{parent}(z) \gets y$,
                \Else $\;$ add $(y,z)$ to $E_0$, 
                \EndIf
            \Else
                $\; x \gets \text{parent}(y)$, 
                \If{$\zeta^x_{(y,z)}=1$} 
                    \State $t_I(z) \gets t_I(y)+T^+_{\{y,z\}}+X_{(y,z)}$, add $z$ to $S$,
                    \State add $(t_I(z),z,+)$ to $\text{events}$, add $(y,z)$ to $E_2$, {$\text{parent}(z) \gets y$},
                \Else $\;$ add $(y,z)$ to $E_0$,
                \EndIf
            \EndIf
        \EndFor  
        \State move $y$ from $S$ to $S_T$
\EndWhile
\end{algorithmic}
\caption{The second dominated process $\underline \eta$ and its infection tree}
\label{ALG-CPDE2}
\end{algorithm}
\end{comment}
As in the previous section, we define
the process $(\underline{\eta}_t)_{t \ge 0}$ by setting
\[
\forall t\ge 0 \quad \forall x\in \Zd \quad x \in \underline{\eta}_t \Longleftrightarrow t_I(x)\le t < t_R(x),
\]
with the convention that $t_I(x)=t_R(x)=\infty$ if these values have not been calculated by the algorithm. Let us now gather some simple observations. We still have properties \ref{OBS:infected},  \ref{OBS:edge1}, \ref{OBS:site1} and \ref{OBS:site} but \ref{OBS:edgerate} is replaced by
\begin{enumerate}
    \item[(P3')]  When, during one passage in the loop, we treat some $y \in S$, we only simulate one infection to each not yet discovered neighbor $z$, with rate $\lambda$: either edge $\{y,z\}$ is available and we simulate the first infection through $(y,z)$, or it is unavailable, and we wait for its next switch to available and we may simulate the next infection through $(y,z)$ after this switch. 
\end{enumerate}
As before, $(\underline{\eta}_t)_{t \ge 0}$  is stochastically dominated by the $\text{CPDE}(p,v,\lambda)$ and the random oriented graph $T=(S_T,E_1\cup E_2)$ is the infection tree of once infected sites in $\underline \eta$. In particular, if $S_T$ grows infinitely, then the simulated process $(\underline{\eta}_t)_{t \ge 0}$ lives forever. 

\subsection*{Second step.} We now build a coupling between the previously simulated infection tree $T$ and the connected component of the origin in a percolation on $\Zd$ with two types of edges.

\medskip
The random variables  $(\xi_{(x,y)})_{(x,y) \in \vec{\E}_d}$ are identically distributed, with values in $\{0,1\}$,  and such that
\[
\P(\xi_{(y,z)}=1)=\P(X_{(y,z)}<\min(X_y,T^-_{\{y,z\}})) = \frac{\lambda}{\lambda+1+v(1-p)}.
\]
If $p\in (0,1)$ and $v>0$ are fixed, this term tends to $1$ as $\lambda$ tends to infinity. As the variables are only locally dependent, and independent from $(\omega_e)_{e \in \E_d}$, we can apply the stochastic comparison result by Liggett, Schonmann and Stacey~\cite{Liggett_DomStoch}:
for every $0<q<1$, we can choose $\lambda$ large enough and by increasing the probability space if necessary, we can consider on the same probability space a fourth family of independent and identically distributed random variables, 
\begin{equation}
\label{EQ:couplingTperco2}
\text{$(\underline \xi_{(x,y)})_{(x,y) \in \vec{\E}_d}$, with law $\mathcal B(q)$, such that } \forall (x,y) \in \vec{\E}_d \quad \underline \xi_{(x,y)} \le \xi_{(x,y)} \; a.s. 
\end{equation}
and independent from $(\omega_e)_{e \in \E_d}$.
%$$
%    \forall (x,y) \in \vec{\E_d} \quad \P(\underline \xi_{(x,y)} \le \xi_{(x,y)})=1.
%$$

\medskip
Let us now study the family $\zeta$. \\
Note that the $(\zeta_{(y,z)}^x)$ are independent from $(\omega_e)_{e \in \Ed}$, and identically distributed. 
Take now $(y,z) \in \vec{\E}_d^s$. As the parent $x$ of $y$ used in the algorithm is random, we naturally introduce:
\[
\forall (y,z) \in \vec{\E}_d^s \quad \underline \zeta_{(y,z)} =\prod_{x \in \mathcal N(y) \backslash \{z\}} \zeta_{(y,z)}^x.
\]
%We also set $\underline \zeta_{(\mathbf{0},\mathbf{0})}=0$. Note that if $(y,z) \in \E_d \backslash \E_d^r$, then $\underline \zeta_{(y,z)}=0$. 
By the choice we made for $\vec{\E}_d^s$, if $(y,z)$ and $(y',z')$ are two distinct edges in $\vec{\E}_d^s$, the set composed of $y$, $z$ and their neighbors is disjoint from the set composed of $y'$, $z'$ and their neighbors. This implies that the random vectors $(\underline \zeta_{(y,z)}, \underline\zeta_{(z,y)})$, indexed by $\{y,z\} \in \E_d$, are independent.

\begin{lemma} 
\label{LEM:zeta}
Fix $d \ge 2$, $p \in (0,1)$, $v>0$ and $0<r<\frac1{2^{2d-1}}\frac{p}{1+(2d-2)(1-p)}$. \\
There exists $\lambda_0>0$ such that for every $\lambda \ge \lambda_0$, for every $(y,z) \in \E_d^r$, 
$$\P(\underline \zeta_{(y,z)}=1) \ge r.$$
\end{lemma}
It will be important in the proof of Theorem \ref{TH-survieSC}, for selecting all the parameters in the correct ordrer, that the upper bound for $r$ in this lemma does not depend on the update speed $v$. 

\begin{proof} 
Fix $d \ge 2$, $p \in (0,1)$ and $v>0$. \\
Take $(y,z) \in \vec{\E}_d^s$ and note $\mathcal N=\mathcal N(y) \backslash \{z\}$ and $T^+=T_{\{y,z\}}^+$. For every $x \in \mathcal N$,
\begin{align*}
\zeta_{(y,z)}^x & = 
\mathbf{1}_{T^+ < X_y}\mathbf{1}_{X_{(y,z)}<\min(X_y-T^+,T^-_{\{y,z\}})}  \\
& \quad + \mathbf{1}_{X_y<T^+} \mathbf{1}_{T^+  <T^-_{\{x,y\}}-X_{(x,y)}}  \mathbf{1}_{ y \in  \tilde \eta^{x,y}_{T^+}} \mathbf{1}_{X_{(y,z)}<\min(X_y',T^-_{\{y,z\}})}, \\
& \ge \mathbf{1}_{T^+  <T^-_{\{x,y\}}-X_{(x,y)}}  \mathbf{1}_{ y \in  \tilde \eta^{x,y}_{T^+}}
  \\
& \quad \times  \left( \mathbf{1}_{T^+ < X_y}\mathbf{1}_{X_{(y,z)}<\min(X_y-T^+,T^-_{\{y,z\}})} + \mathbf{1}_{X_y<T^+} \mathbf{1}_{X_{(y,z)}<\min(X_y',T^-_{\{y,z\}})} \right).
\end{align*}
With the strong Markov property at stopping time $T^+$, we have
\begin{align*}
\P(\underline \zeta_{(y,z)}=1) 
& \ge  \P \left( \bigcap_{x \in \mathcal N} \left\{T^+ < T_{\{x,y\}}^-, \; y \in  \tilde \eta^{x,y}_{T^+}  \right\} \right) \times \P \left( X_{(y,z)} < \min(X_y,T_{\{y,z\}}^-)  \right) \\
&= A_\lambda \times B_\lambda.
\end{align*}
For the second term, we again use that
\begin{align}
    B_\lambda=\P\left( X_{(y,z)} < \min(X_y,T_{\{y,z\}}^-) \right)  = \frac{\lambda}{\lambda + 1 + v(1-p)}\stackrel{\lambda \to +\infty}{\longrightarrow} 1. \label{Bbound}
\end{align} 
Let us deal with the first term,
\begin{align*}
A_\lambda 
& = \P \left( 
 \left\{T^+ < \min_{x \in \mathcal N} T_{\{x,y\}}^-\right\} 
\cap \left\{ y \in \bigcap_{x \in \mathcal N} \tilde \eta^{x,y}_{T^+} \right\} \right). 
\end{align*}
Note that $(\tilde \eta_t^{x,y})_{t \ge 0,x \in \mathcal N}$ and $(T_{\{x,y\}}^-)_{x \in \mathcal N}$ are independent. Thus, conditioning on the $\sigma$-algebra generated by $T^+$, we get:
\begin{align*}
A_\lambda & =  \E \left[ 
\P \left( \left. T^+ < \min_{x \in \mathcal N} T_{\{x,y\}}^- \right| T^+ \right) 
\P \left( \left.  y \in \bigcap_{x \in \mathcal N} \tilde \eta^{x,y}_{T^+}  \right| T^+ \right) \right].
\end{align*}
Next, note that the sets $\eta^{x,y}_{T^+}$, for $x \in \mathcal N$, are non-increasing
with respect to configurations of the Poisson process of recoveries for $y$ during $[0, T^+]$ (and this is their only source of dependence), so we can apply the FKG inequality. We take an arbitrary $x_0 \in \mathcal N$, and we have
%{ref: d'après thèse Aurelia, Meester Roy 96}
\begin{align*}
A_\lambda &  \ge \E \left[ 
\P \left( \left. T^+ < \min_{x \in \mathcal N} T_{\{x,y\}}^- \right| T^+ \right) 
\P \left( y \in \tilde \eta^{x_0,y}_{T^+} | T^+\right)^{2d-1}
\right].
\end{align*}
We denote by $\tau_{(x_0,y)}$ the extinction time of $(\tilde \eta_t^{x_0,y})_{t \ge 0}$. \\
By symmetry, $\P ( y \in \tilde \eta^{x_0,y}_{T^+} | T^+) \ge \frac12 \P ( T^+ <\tau_{(x_0,y)} | T^+)$, so
\begin{align*}
A_\lambda &  \ge \frac1{2^{2d-1}} \E \left( 
\P \left( \left. T^+ < \min_{x \in \mathcal N} T_{\{x,y\}}^- \right| T^+ \right) 
\left( \P ( T^+ <\tau_{(x_0,y)} | T^+) \right)^{2d-1}
\right).
\end{align*}
Recall that $(\tilde \eta_t^{x_0,y})_{t \ge 0}$ starts from both sites infected. Whenever it contains only one infected site, it has a probability $\frac{1}{\lambda+1}$ to jump to the empty state on its next move. When it is in the full state, the delay before the next jump follows an exponential distribution with parameter $2$. So its extinction times stochastically dominates the random variable
\[
T=\sum_{i=1}^N X_i,
\]
where $N$ is a geometric random variable with parameter $\frac{1}{\lambda+1}$, $(X_i)_{i \ge 1}$ are exponential random variables with parameter $2$, all independent. 
Consequently, $T$ follows the exponential law with parameter $\frac{1}{2(\lambda+1)}$, so we have

\begin{align*}
  \forall t>0  \quad   \P(\tau_{(x_0,y)} \ge t)   \ge \P( T \ge t) &= \exp \left( - \frac{t}{2(\lambda+1)}\right) \stackrel{\lambda \to +\infty}{\longrightarrow} 1.
\end{align*}
As $T^+<+\infty$ almost surely, we can write
\begin{align*}
    \left( \P ( T^+ <\tau_{(x_0,y)} | T^+) \right)^{2d-1} & \ge \exp \left( - \frac{(2d-1)T^+}{2(\lambda+1)}\right)  \stackrel{\lambda \to +\infty}{\longrightarrow} 1 \; a.s.
\end{align*}
%and the term on the right converges almost surely to 1. 

By dominated convergence, we obtain
\begin{align}
\liminf_{\lambda \to +\infty} A_\lambda 
& \ge \frac1{2^{2d-1}} \E \left( 
\P \left( \left. \min_{x \in \mathcal N} T_{\{x,y\}}^- >T^+  \right| T^+ \right) \right) \notag\\
& = \frac1{2^{2d-1}} \P \left( \min_{x \in \mathcal N} T_{\{x,y\}}^- >T^+ \right) \notag\\
& =\frac1{2^{2d-1}}\frac{vp}{vp+(2d-1)v(1-p)} 
 = \frac1{2^{2d-1}}\frac{p}{1+(2d-2)(1-p)}.\label{Abound}
\end{align} 
Recall that $d,p$ and $v$ have been fixed at the beginning of the proof. Now, let us choose $r<\frac1{2^{2d-1}}\frac{p}{1+(2d-2)(1-p)}$. Using Equations~\eqref{Abound} and~\eqref{Bbound}, we can choose $\lambda$ large enough such that 
\[
\P(\underline \zeta_{(y,z)}=1)=A_\lambda\times B_\lambda \ge r,
\]
 which ends the proof.
\end{proof}

\smallskip
Finally, we introduce random variables to encode  percolation with two types of edges in $\Zd$. 
For every $(y,z) \in \vec{\E}_d$, for every $x \in\mathcal{N}(y)\setminus\{z\}$, we set
\begin{align}
    \label{EQ:zeta}
\epsilon_{(y,z)}^x  & = \xi_{(y,z)}\omega_{\{y,z\}} + 2 \mathbf{1}_{(y,z) \in \vec{\E}_d^s} (1-\omega_{\{y,z\}})\zeta_{(y,z)}^x, \\
\underline \epsilon_{(y,z)}  & =  \underline \xi_{(y,z)}\omega_{\{y,z\}} + 2\mathbf{1}_{(y,z) \in \vec{\E}_d^s}(1-\omega_{\{y,z\}}) \underline \zeta_{(y,z)}, \nonumber
\end{align}
and we have almost surely, for every $x \in \mathcal N(y) \backslash\{z\}$:
\[ 
\underline \epsilon_{(y,z)}  \le \epsilon_{(y,z)}^x.
\]
By construction, the random vectors $(\underline \epsilon_{(y,z)}, \underline\epsilon_{(z,y)})$, indexed by $\{y,z\} \in \E_d$, are independent,
and the law of $\underline \epsilon_{(y,z)}$ takes values in $\{0,1,2\}$ and is different whether $(y,z) \in \vec{\E}_d^s$ or not.

\smallskip
We can now, via the following Algorithm \ref{ALG-CPDE2-couplagePerco}, build a coupling between the previous set $S_T$ of once infected sites in $\underline \eta$ and the connected component of the origin of an independent Bernoulli percolation on edges of $\Zd$ with two types of edges.

\begin{algorithm}[H]
\begin{algorithmic}[1]
\Require{$p,v,\lambda$}
\State $S \gets \{0\}$, $S_T \gets \varnothing$,  $E_0 \gets 0$, $E_1 \gets 0$, $E_2 \gets 0$, $\text{parent}(0) \gets 0$,
\State $\underline S \gets \{0\}$, $\underline S_T \gets \varnothing$,  $\underline E_0 \gets 0$, $\underline E_1 \gets 0$, $\underline E_2 \gets 0$, 
\While{$S \cup \underline S \ne \varnothing$}
    \State take $y$ in $S \cup \underline S$ with some arbitrary deterministic rule, 
    \If{$\textcolor{blue}{y \in S}$}   
        \State $\textcolor{blue}{x \gets \text{parent}(y)}$,
        \For{\textcolor{blue}{$z \in \mathcal N(y)$ with $z \notin S \cup S_T$}}
            \State $\textcolor{blue}{i \gets \varepsilon_{(y,z)}^x}$,
            \textcolor{blue}{add $(y,z)$ to $E_i$},
            \If{\textcolor{blue}{$i \ge 1$}} \textcolor{blue}{$z$ to $S$, $\text{parent}(z) \gets y$},
            \EndIf
        \EndFor   
        \State \textcolor{blue}{move $y$ from $S$ to $S_T$},
    \EndIf
    \If{$y \in \underline S$}    
        \For{$z \in \mathcal N(y)$ with $z \notin \underline S \cup \underline S_T$}
            \State $i \gets \underline \varepsilon_{(y,z)}$,
            add $(y,z)$ to $\underline E_i$,
            \If{$i \ge 1$} $z$ to $\underline S$,
            \EndIf
        \EndFor   
        \State move $y$ from $\underline S$ to $\underline S_T$
    \EndIf
\EndWhile
\end{algorithmic}
\caption{Coupling between the set $S_T$ of once infected sites and independent Bernoulli percolation with two types of edges}
\label{ALG-CPDE2-couplagePerco}
\end{algorithm}

As before, note that if we omit the first "if" block (in blue), this simulates a Markovian construction and exploration algorithm of the connected component $\underline S_T$ of the origin in the percolation with two types of edges given by $\underline \varepsilon$. Moreover, $T=(S_T, E_1 \cup E_2)$ is the infection tree of the process $\underline \eta$ defined in Algorithm \ref{ALG-CPDE2}, and  $\underline T =(\underline S_T, \underline E_1 \cup \underline E_2)$ is a spanning tree of the connected component of the origin in the percolation with two types of edges. 

Let us now check the following property of the coupling: at each step, 
\begin{equation}
\label{EQ:incluBis}
    \underline S_T \subset S_T \text{ and } \underline S \cup \underline S_T \subset S \cup S_T.
\end{equation}

We proceed by recurrence. It is obviously true before entering the while loop.
Let us assume now that $\underline S_T \subset S_T \text{ and } \underline S \cup \underline S_T \subset S \cup S_T$ when we begin a passage in the while loop. We take $y \in S\cup \underline S$.
\begin{itemize}
    \item First case: $y \in S\cap \underline S$. Consider $z \in \mathcal{N}(y)$.
    As $\underline S \cup \underline S_T \subset S \cup S_T$, we have two cases:
    \begin{itemize}
    \item either $z \notin S \cup S_T$, and thus $z \notin \underline S \cup \underline S_T$. The edge $(y,z)$ is examined twice. As 
    $$\P(\forall x \in \mathcal N(y) \backslash\{z\}, \; \underline \epsilon_{(y,z)} \le \epsilon_{(y,z)}^x)=1,$$ 
    if $z$ is added to $\underline S$, it is also added to $S$, which preserves \eqref{EQ:incluBis}.
    \item Or $z \in S \cup S_T$ and $z \notin \underline S \cup \underline S_T$. The site $z$ may be added to $\underline S$, but as it is already in $S \cup S_T$,  \eqref{EQ:incluBis} is preserved.
    \end{itemize}
    At the end of the passage, the two moves of $y$ from $S$ to $S_T$ and from $\underline S$ to $\underline S_T$ preserve \eqref{EQ:incluBis}.
    \item Second case: $y \in S$ and $y \notin \underline S$. The neighbors of $y$ may only be added to $S$, which preserves \eqref{EQ:incluBis}. At the end of the passage, the move of $x$ from $S$ to $S_T$ preserves \eqref{EQ:incluBis}.
    \item Third case: $y \notin S$ and $y \in \underline S$. As $\underline S \cup \underline S_T \subset S \cup S_T$, $y \in S_T$. Consider $z \in \mathcal{N}(y)$ and imagine it is added to $\underline S$: so $\underline \epsilon_{(y,z)}=1$, which implies $\epsilon_{(y,z)}^x=1$ where $x$ is the parent of $y$. As $y \in S_T$, this implies that $z$ is already in $S\cup S_T$ and its addition to $\underline S$ preserves~\eqref{EQ:incluBis}. At the end of the passage, $y$ is moved $\underline S$ to $\underline S_T$. As $y \in S_T$, this moves  preserves \eqref{EQ:incluBis}.
\end{itemize}
This ends the proof of \eqref{EQ:incluBis}.

\smallskip
Let us come back to the percolation with two types of edges given by $\underline \varepsilon$.

\begin{defi}
  For any integer $d \ge 2$ and any real numbers $ a, b\ge 0$ with $ a+b \le 1$, let us denote by $\mathrm{Perco}(d,a,b)$ the percolation model on the edges $\E_d$ of $\Zd$ where
\begin{itemize}
\item the states of edges are independent; 
\item edges in $\E_d \backslash \vec{\E}_d^s$ open with probability $a$, and closed with probability $1- a$; 
\item edges in $\vec{\E}_d^s$ open with probability $ a+b $, and closed with probability $1- a-b$. 
\end{itemize}  
\end{defi}

A straightforward adaptation of the proof by Aizenman and Grimmett \cite{AizenmanGrimmett91}
ensures that $\mathrm{Perco}(d,a,b)$ with $b>0$ percolates more easily than the standard percolation on the edges $\E_d$ of $\Zd$ with parameter $a$:
\begin{lemma}[Aizenman-Grimmett] 
\label{Enhancement2}
Fix $d \ge 2$. For every $b_0 \in (0,1-p_c(d))$, there exists $a_0<p_c(d)$ such that for every $a\ge a_0$, for every $b \ge b_0$ such that $a+b \le  1$, the cluster containing the origin in $\mathrm{Perco}(d,a,b)$ is infinite with positive probability. 
\end{lemma}

\medskip
We can now conclude the proof of Theorem~\ref{TH-survieSC}. Fix $d \ge 2$, take $r$ such that 
\[
0 < r < {\frac1{2^{2d-1}}}\frac{p_c(d)}{1+(2d-2)(1-p_c(d))}
\]
and fix $b_0=\min\left(\frac{1-p_c(d)}2, r(1-p_c(d))\right)$, so  that $b_0<1-p_c$. \\
Take $a_0<p_c(d)$ associated to $b_0$ by Lemma \ref{Enhancement2}. By continuity, we can then choose $p \in (a_0,p_c(d))$, close enough to $p_c(d)$, such that
\[
r<{\frac1{2^{2d-1}}} \frac{p}{1+(2d-2)(1-p)}.
\]
Fix $v>0$. By Lemma \ref{LEM:zeta}, we choose $\lambda$ large enough such that for every $(y,z) \in \vec{\E}_d^s$, $\P(\underline \zeta_{(y,z)}=1) \ge r$, so that
$$
 b  =\P(\underline \epsilon_{(y,z)}=2)=
\P(\underline \zeta_{(y,z)}=1)\P(\omega_{\{y,z\}}=0) \ge r(1-p) \ge r(1-p_c(d)) \ge b_0 .
$$
Finally, choose $q \in (0,1)$ such that $p q \ge a_0$. 
With \eqref{EQ:couplingTperco2}, by increasing $\lambda$ if necessary,  for every $(y,z) \in \vec{\E}_d$, $\P(\underline \xi_{(y,z)}=1) = q$, so that
$$ a= \P(\underline \epsilon_{(y,z)}=1)=
\P(\underline \xi_{(y,z)}=1)\P(\omega_{\{y,z\}}=1) = pq \ge a_0.$$

With Algorithm \ref{ALG-CPDE2-couplagePerco}, we have constructed the random set $\underline{S}_T$ which is the connected component of the origin in $\mathrm{Perco}(d,a,b)$. By Lemma \ref{Enhancement2}, this percolation is supercritical, so with positive probability, $\underline S_T$ is infinite. The coupling inclusion~\eqref{EQ:incluBis} ensures that $S_T$ is also infinite with positive probability. Since, by construction of Algorithm \ref{ALG-CPDE2}, the set $S_T$ is the set of sites infected once in the process $\underline \eta$, this process lives forever with positive probability. By stochastic comparison, the $\text{CPDE}(p,v,\lambda)$ lives forever with positive probability. This ends the proof of Theorem~\ref{TH-survieSC}.

\bibliographystyle{alpha}

\end{document}